\documentclass[a4paper,english,openright,11pt]{smfartVScomptage}

\usepackage{smfenum,smfthm,amsmath,amssymb,amscd,mathrsfs,euscript,color}
\usepackage{stmaryrd,mathabx,upgreek}
\usepackage[latin1]{inputenc}
\input{xypic}

\numberwithin{equation}{section} 
\theoremstyle{plain}

\newcounter{nonum}

\def\AA{\mathbb{A}}
\def\CC{\mathbb{C}}

\def\PP{\mathbb{P}}
\def\QQ{\mathbb{Q}}

\def\ZZ{\mathbb{Z}} 
\def\FF{\mathbb{F}} 
\def\WW{\mathbb{W}} 

\def\A{{ A}}

\def\D{{ D}}

\def\F{{ F}}
\def\G{{ G}}

\def\I{{ I}}

\def\K{{ K}}
\def\L{{ L}}
\def\M{{ M}}
\def\N{{ N}}
\def\P{{ P}}

\def\SS{{ S}}
\def\T{{ T}}
\def\U{{ U}}

\def\W{{ W}}

\def\Z{{ Z}}

\def\Cc{\EuScript{C}}

\def\Kk{\EuScript{K}}

\def\Oo{\EuScript{O}}

\def\Ww{\EuScript{W}}

\def\Ga{\Gamma}
\def\La{\Lambda}
\def\Om{\Omega}

\def\a{\alpha} 
\def\b{\beta}

\def\g{\gamma}
\def\h{\varphi}
\def\k{\kappa}

\def\m{\mathfrak{m}}
\def\n{\eta}

\def\p{\mathfrak{p}}
\def\s{\sigma}

\def\w{\varpi}

\def\rp{\rangle}
\def\>{\geqslant}
\def\<{\leqslant}

\def\Hom{{\rm Hom}}

\def\GL{{\rm GL}}
\def\Gal{{\rm Gal}}
\def\Ker{{\rm Ker}}

\def\Ind{{\rm Ind}}
\def\ind{{\rm ind}}

\def\qlb{{\overline{\QQ}_\ell}}
\def\zlb{{\overline{\ZZ}_\ell}}
\def\flb{{\overline{\FF}_{\ell}}}

\def\ip{\boldsymbol{i}}
\def\rp{\boldsymbol{r}}

\def\r{{\textbf{\textsf{r}}}}

\def\dd{{\rm d}}

\def\ii{\iota}

\def\GG{{\bf G}}

\long\def\MSC#1\EndMSC{\def\arg{#1}\ifx\arg\empty\relax\else
     {\par\narrower\noindent%
     2010 Mathematics Subject Classification: #1\par}\fi}

\long\def\KEY#1\EndKEY{\def\arg{#1}\ifx\arg\empty\relax\else
	{\par\narrower\noindent Keywords and Phrases: #1\par}\fi}

\title{On modular rigidity for $\GL_n$}

\author{Nadir Matringe}
\address{Institute of Mathematical Sciences,
  NYU Shanghai,
  3663 Zhongshan Road North Shanghai,
  200062, China
  \&
  Institut de Math\'ematiques de Jussieu-Paris Rive Gauche,
  Universit\'e Paris Cit\'e,
  75205,
  Paris,
  France}
\email{nrm6864@nyu.edu, matringe@img-prg.fr}

\author{Alberto M\'{\i}nguez}
\address{
University of Vienna, 
Fakult{\"a}t f{\"u}r Mathematik,
Oskar-Morgenstern-Platz 1,
1090 Wien}
\email{alberto.minguez@univie.ac.at}

\author{Vincent S\'echerre}
\address{
Laboratoire de Math\'emati\-ques de Versailles, 
UVSQ, 
CNRS, 
Universit\'e Paris-Saclay,
78035 Versailles, France}
\email{vincent.secherre@uvsq.fr}

\begin{abstract}
Let $k$ be a global field and $\AA_{k}$ be its ring of ad\`eles.
Let $\ell$ be a prime number and~fix a field isomorphism from $\CC$ to 
$\qlb$. 
Let $\Pi_1$, $\Pi_2$ be cuspidal automorphic representations of 
$\GL_n(\AA_{k})$ for some integer $n\>1$.
In this paper, we study the following question:
assuming that~there~is~a~fi\-ni\-te~set $S$ of places of $k$ containing
all Archimedean places and all finite places above $\ell$ 
such~that,
for~all $v\notin S$,
the~local~components $\Pi_{1,v} \otimes_{\CC} \qlb$ and
$\Pi_{2,v} \otimes_{\CC} \qlb$ are
unramified and their Satake~para\-meters are congruent~mod~$\ell$,
are the local~components $\Pi_{1,w} \otimes_{\CC} \qlb$ and
$\Pi_{2,w} \otimes_{\CC} \qlb$ integral,
and~do their~reductions mod $\ell$ share an irreducible factor 
for all non-Archimedean places $w$ not dividing~$\ell$?
We show that,
under certain~conditions on $\Pi_1$, $\Pi_2$,
the answer is yes. 
We also give a simple proof when $k$ is a function field. 
\end{abstract} 

\begin{document} 

\maketitle

\MSC 
\EndMSC
\KEY 
Automorphic forms,
Congruences mod $\ell$,
Satake parameters,
Whittaker models, 
Automorphic representations
\EndKEY

\thispagestyle{empty}

\section{Introduction}

\subsection{}

Let $k$ be a number field and $\AA_{k}$ be its ring of ad\`eles.
Let $\Pi_1$ and $\Pi_2$ be cuspidal automorphic representations of
$\GL_n(\AA_{k})$ for some integer $n\>1$.
The rigidity
(or strong multiplicity~$1$)~theo\-rem asserts that,
if there is a finite set $S$ of places of $k$ containing all 
Archimedean places~such that, for all $v\notin S$, 
the local components $\Pi_{1,v}$ and $\Pi_{2,v}$ are unramified and 
have the same~Sata\-ke para\-me\-ter, 
then~$\Pi_1$ and $\Pi_2$ are isomorphic (\cite{PS,Cogdell,JS1,JS2}).
A similar result holds over function fields. 

\subsection{}

Now fix a field isomorphism $\iota$ from $\CC$ to an algebraic closure $\qlb$
of the field of $\ell$-adic~num\-bers for some prime number $\ell$,
and consider the collections of irreducible smooth~$\qlb$-representa\-tions
of $\GL_n(k_v)$ defined by
\begin{equation}
\label{gremond}
\pi_{i,v} = \Pi_{i,v} \otimes_{\CC} \qlb, 
\quad
i\in\{1,2\},
\end{equation}
where the tensor product is taken with respect to $\iota$, 
$v$ runs over all finite places of $k$
and $k_v$~is~the completion of $k$ at $v$.
As $\Pi_1$ and $\Pi_2$ are cuspidal,
these representations are generic (see \S\ref{defgen}). 

Suppose that there exists a finite set $S$ of places of $k$ containing
all~Archi\-me\-dean places and~all fi\-ni\-te places above $\ell$ such that,
for all $v\notin S$,
the following are satisfied:
\begin{enumerate}
\item 
the representations $\pi_{1,v}$ and $\pi_{2,v}$ are unramified
representations of $\GL_n(k_v)$,
\item
the Satake parameters $\s_{1,v}$ and $\s_{2,v}$ of these unramified
representations,
considered as~con\-jugacy classes of~semi\-simple elements of $\GL_n(\qlb)$,
have  their  characteristic  polynomials   $P_{1,v}(X)$  and  $P_{2,v}(X)$  in
$\zlb[X]$,
where $\zlb$ is the ring of integers of $\qlb$,
\item
the reductions of $P_{1,v}(X)$ and
$P_{2,v}(X)$ in $\flb[X]$ are equal,
$\flb$ being the residue field of $\zlb$.
\end{enumerate}

Assumption 2 is equivalent to saying that
the unramified representations $\pi_{1,v}$ and $\pi_{2,v}$ are~in\-te\-gral, 
that is, 
their $\qlb$-vec\-tors spa\-ces 
contain $\GL_n(k_v)$-stable $\zlb$-lattices~(see \S\ref{par231}).
One can~then consider their reductions mod~$\ell$,
denoted $\r_\ell(\pi_{1,v})$~and $\r_\ell(\pi_{2,v})$,
which are finite length, semisimple smooth $\flb$-representations 
of $\GL_n(k_v)$
(see Section \ref{henriette} for a precise definition of reduction mod~$\ell$).
As\-sump\-tion~3 is then equivalent to saying that the representations
$\r_\ell(\pi_{1,v})$ and $\r_\ell(\pi_{2,v})$ are equal
(see Remarks \ref{reductiontowinS} and \ref{redunramnongeneric}).

Now let $w$ be a finite place of $k$ not dividing $\ell$. 
Our first question is

\begin{enonce}{Question}
\label{Q1}
Are the irreducible representations $\pi_{1,w}$ and $\pi_{2,w}$ integral? 
\end{enonce}

Assume that this is the case. 
One can then form $\r_\ell(\pi_{1,w})$~and $\r_\ell(\pi_{2,w})$.
These representations may not be equal
(see Remark \ref{counterex} for an example),
but one may address the following~question.

\begin{enonce}{Question}
\label{Q2}
Do $\r_\ell(\pi_{1,w})$~and $\r_\ell(\pi_{2,w})$
have an irreducible component in common?
\end{enonce}

If $k$ is a totally real (respectively, CM) number field, 
and if $\Pi_1$, $\Pi_2$ are algebraic regular,~self\-dual
(respectively, conjugate-selfdual)
cuspidal automorphic representations,
then \cite{MSautomodl} Theorem 8.2 says that the answers to
Questions \ref{Q1} and \ref{Q2} are yes.
More precisely:
\begin{itemize}
\item 
the representations $\pi_{1,w}$ and $\pi_{2,w}$ are integral for all finite 
places $w$ of $k$ not dividing $\ell$, 
\item
their reductions mod $\ell$ have a unique generic irreducible component in 
common, 
\item
this unique common generic irreducible component occurs with 
multiplicity $1$.
\end{itemize}
Such a~result,
which can be thought of as a modular rigidity theorem,
has been used in \cite{MSautomodl}~in~or\-der to study the beha\-vior~of
local transfer for cuspidal $\qlb$-representations of quasi-split classical groups
with respect~to congruences mod $\ell$.

More generally, 
thanks to the results of \cite{HLTT,Scholze,Varma},
one can make the argument of the~proof of~\cite{MSautomodl} Theorem 8.2 
work with no duality assumption on $\Pi_1$ and $\Pi_2$:
if $k$ is a total\-ly real or~CM number field, and
if~$\Pi_1$, $\Pi_2$ are algebraic~regu\-lar, cuspidal automorphic 
representations, 
the~answers~to~Ques\-tions \ref{Q1} and \ref{Q2}
are still~yes; more~preci\-sely, the three 
properties above still~hold. (See~\S 
\ref{HLTTScholzeVarma} below~for a detailed 
argument,~which relies on the existence of a correspondence from 
algebraic regular~cus\-pidal automorphic representation to
Galois representations
with local-global compatibility at all finite places not dividing $\ell$.)

It is natural to ask whether the `totally real or CM' assumption on $k$, 
or the `algebraic regular' assumption on the representations $\Pi_1$ and $\Pi_2$,
or the cuspidality assumption,
can~be removed.
We~will~not investigate these questions in the present article. 

It is also natural to seek an elementary,
purely automorphic proof of such a modular rigidity theo\-rem,
avoiding the use of Galois representations and local-global
compatibility~theorems.~We will
study this~question~in the case of function fields, which is easier
since there are no~Archi\-me\-dean places.

\subsection{}

We now assume that $k$ is a function field of characteristic $p$,
with ring of ad\`eles $\AA_k$.~Recall
that we have fixed a field isomorphism $\iota$ from $\CC$ to $\qlb$ 
for some prime number $\ell$ which we~assume to be different from $p$.
In this article, we prove the following theorem
(see Theorem \ref{MAINTHM}).

\begin{theo}
\label{MAINTHMintro}
Let $\Pi_1$ and $\Pi_2$ be cuspidal~automor\-phic representations
of $\GL_n(\AA_{k})$.
Associa\-ted with them, there are the representations $\pi_{i,v}$
defined by \eqref{gremond}.
Suppose that there exists a~finite set $\SS$ of places of $k$~such
that, for all $v\notin\SS$, one has:
\begin{enumerate}
\item
the representations $\pi_{1,v}$ and $\pi_{2,v}$ are unramified,
\item
the characteristic polynomials of their Satake parameters are in $\zlb[X]$
and have the same reduction in $\flb[X]$.
\end{enumerate}
Let $w$ be a place of $k$. 
Then
\begin{itemize}
\item 
the representations $\pi_{1,w}$ and $\pi_{2,w}$ are integral,
\item
their reductions mod $\ell$ share a  generic irreducible compo\-nent,
\item
and such a generic~com\-ponent is unique and occurs with multiplicity $1$ 
in both reductions. 
\end{itemize} 
\end{theo}

This theorem can be easily deduced from L.~Lafforgue's
global Langlands correspondence~\cite{Lafforgue}
(see Remark \ref{RemLaf}).
Our purpose is to give a simple proof of Theorem \ref{MAINTHMintro}
which does not rely~on the Langlands corres\-pondence for function fields.
Our argument,
inspired from
Piatetski-Shapi\-ro's proof of the classical rigidity theorem 
\cite{PS,Cogdell} is described below.
We currently~do~not~know how to extend our argument to number fields. 

\subsection{}
\label{sidonierougon}

Before explaining the proof of Theorem \ref{MAINTHMintro},
we introduce our main local~ingre\-dients. 
Let $F$ be a non-Archimedean locally compact field of
residue characteristic $p$ (and~cha\-rac\-teristic $0$~or
$p$).~Fix a non-trivial smooth $\qlb$-character $\vartheta$ of $F$.

\begin{prop}
\label{proplocalintro}
Let $\pi_1$ and $\pi_2$ be integral generic irreduci\-ble
$\qlb$-representa\-tions of~$\GL_n(F)$.
Suppose that
there are functions $\W_1$ and $\W_2$
in the Whittaker models of $\pi_1$ and $\pi_2$ with respect~to $\vartheta$
satisfying the following conditions:
\begin{enumerate}
\item
$\W_1$ and $\W_2$ are $\zlb$-valued and $\W_1(1)=\W_2(1)=1$,
\item
the reductions of $\W_1(g)$ and $\W_2(g)$ in $\flb$ are equal for all $g\in\G$. 
\end{enumerate}
Then $\r_\ell(\pi_1)$ and $\r_\ell(\pi_2)$ share a 
generic irreducible component,
such a generic irreducible compo\-nent is unique
and it occurs with multiplicity $1$.
\end{prop}

Let $P$ be the mirabolic subgroup of $\GL_n$,
made of all matrices with~last~row $(0\ \dots\ 0\ 1)$,~and $N$
be its unipotent radical. 
We have the following remarkable~integrality~crite\-rion 
(see Proposi\-tion~\ref{lemHM}), 
which follows from \cite{HelmMossINVENT18} and \cite{MaMo}.

\begin{prop} 
\label{integralschwartz}
Let $\pi$ be a generic irreducible $\qlb$-representa\-tion of $\GL_n(F)$.
The following assertions are equivalent.
\begin{enumerate}
\item 
The representation $\pi$ is integral. 
\item
Given any function in the Whittaker model of $\pi$ with respect to $\vartheta$
whose restriction to~$P(F)$ is compactly supported mod $N(F)$,
this function is $\zlb$-valued on $\GL_n(F)$ 
if and only if it is $\zlb$-va\-lued on $P(F)$.
\end{enumerate}
\end{prop}

\subsection{}
\label{aristidemacquart}

We now introduce our main global ingredients.
Fix a continuous unitary complex charac\-ter $\psi$ of $\AA_k$,
and consider a cuspidal {automorphic form} $\h$ on $\GL_n(\AA_k)$.
Associated with it by~\eqref{psiw},
with respect to the choice of $\psi$,
there is~a~Whit\-ta\-ker function $W$ on $\GL_n(\AA_k)$.
Note that $\psi$ is valued in the group
$\mu_{p}$ of complex $p$th~roots~of $1$.
Let $Z$ be the centre of $\GL_n$.

Given any sub-$\ZZ[\mu_{p}]$-module $\I$ of $\CC$,
we prove in Section \ref{SEC3} that:
\begin{itemize}
\item 
if $W$ takes values in $I$ on $P(\AA_k)$, 
then $\h$ takes values in $I$ on $P(\AA_k)$
(Theorem \ref{GHthm}),
\item
if $\h$ takes values in $I$ on $Z(\AA_k)P(\AA_k)$, 
then $W$ takes values in $I$ on $\GL_n(\AA_k)$
(Theorem \ref{MSthm}). 
\end{itemize}
% We now explain how to use the results of Paragraphs \ref{sidonierougon}
% and \ref{aristidemacquart} to prove our main theorem \ref{MAINTHMintro}. 

\subsection{}
\label{reineclaude}

We now consider two cuspidal~automor\-phic representations
$\Pi_1$, $\Pi_2$ of $\GL_n(\AA_{k})$ as in~Theo\-rem \ref{MAINTHMintro}.
Let $A$ be the local ring $\iota^{-1}(\zlb)$ and $\m$ be its maximal ideal.
Fix a place $w\in S$.

We first observe that
the central characters of $\Pi_1$ and $\Pi_2$ 
are $\A$-valued and congruent mod $\m$
(see Pro\-po\-sition \ref{PropAnimus}),
thanks to the information we have at all places $v\notin S$.

For each place $v$ of $k$,
let $\psi_v$ be the local component of $\psi$ at $v$.
It is a smooth character~of~$k_v$,
which we may assume,
for all $v\notin S$,
to be trivial on the ring of integers of $k_v$ but not on~the~in\-ver\-se of its
maximal ideal.

Let $W_{i,v}$ be any function in the~Whit\-ta\-ker~model~of~$\Pi_{i,v}$
with~respect to $\psi_v$ such that:
\begin{enumerate}
\item 
if $v\notin\SS$,
then $W_{i,v}$ is $\GL_n(\Oo_v)$-invariant
(see \S \ref{Wpi}),
\item
if $v\in\SS$, 
then $W_{1,v}$ and $W_{2,v}$ coincide on $\P(k_v)$,
and their restriction to $\P(k_v)$~is~com\-pactly supported mod $N(k_v)$
and $A$-valued (see \S \ref{defgen}),
\item 
and $W_{1,v}(1) = W_{2,v}(1) = 1$ for all $v\neq w$.
\end{enumerate}
The tensor product of the $W_{i,v}$ is
a function $W_i$ in the Whittaker model of $\Pi_i$ with respect~to~$\psi$. 
First, 
it follows from the Shintani formula (see \cite{Shintani,Cogdell}
and Propo\-sition \ref{pagniez})
that, for all $v \notin S$:
\begin{itemize}
\item 
the functions $W_{1,v}$ and $W_{2,v}$ are $A$-valued, 
\item 
the difference $W_{1,v} - W_{2,v}$ is $\m$-valued.
\end{itemize} 
The functions $W_1$ and $W_2$ are thus $A$-valued and 
$W_1-W_2$ is $\m$-va\-lued~on $\Z(\AA_k)\P(\AA_k)$. 

Since $\A$ and $\m$ are sub-$\ZZ[\mu_{p}]$-modules of $\CC$,
we may apply the result~of~Pa\-ra\-graph~\ref{aristidemacquart},
from which we deduce that the functions $W_1$ and $W_2$ are $A$-va\-lued
and $W_1-W_2$ is $\m$-va\-lued on the whole of $\GL_n(\AA_k)$. 
Consequently,
thanks to Assumption (3) above, we get:
\begin{itemize}
\item[$\phantom{\star}(\star)$]
the functions $W_{1,w}$ and $W_{2,w}$~are $A$-va\-lued, % on $\GL_n(k_w)$, 
\item[$(\star\star)$]
the difference $W_{1,w}-W_{2,w}$ is $\m$-valued. 
\end{itemize}
Starting with any function in the~Whittaker model of~$\pi_{i,w}$
with respect~to~$\psi_w$
whose~restriction to~$P_w$ is compactly supported mod $N_w$
and $\zlb$-valued,
we thus proved 
that this func\-tion is~$\zlb$-va\-lued (see $(\star)$ above).
Applying Pro\-po\-si\-tion~\ref{integralschwartz}, 
we deduce that $\pi_{1,w}$ and $\pi_{2,w}$~are~integral,~pro\-ving~the first 
assertion of Theorem \ref{MAINTHMintro}. 

Now assume that $W_{1,w}$ and $W_{2,w}$ satisfy the additional condition 
$W_{1,w}(1) = W_{2,w}(1) = 1$.~The
remaining two assertions of Theorem \ref{MAINTHMintro} then follow from
$(\star)$ and $(\star\star)$ by Pro\-po\-sition \ref{proplocalintro}.

\section*{Acknowledgments}

We would like to thank 
Guy~Henniart,
Harald~Grobner, 
Rob Kurinczuk, 
Gil Moss and
Hongjie Yu
for stimulating~dis\-cus\-sions about this work.

This work was partially supported by 
the Erwin Schr\"odinger Institute in Vienna,
when we~be\-nefited from a 2020 Research in Teams grant.
We thank the institute for hospitality, 
and for~ex\-cel\-lent working conditions.

The research of A. M\'inguez was partially funded by the
Principal Investigator
project PAT\-4832423 of the Austrian Science Fund (FWF).

V.~S\'echerre also thanks the Institut Universitaire de France
for support and ex\-cel\-lent working conditions
when this work was done.

\section{Local considerations}
\label{henriette}

In this section, 
$\F$ denotes a locally compact non-Archimedean field of residue characteristic 
$p$ and $n\>1$ is a positive integer.
We write $\Oo_\F$ for the ring of integers of $\F$ and $\p_\F$ for its maximal 
ideal.
We also write $\G$ for the locally profinite group $\GL_n(\F)$.

Let $\ell$ be a prime number different from $p$.
We write $\qlb$ for an algebraic closure of the field of $\ell$-adic integers,
$\zlb$ for its ring of integers and $\flb$ for its residue field. 

Let $\psi:\F\to\overline{\QQ}{}^\times_\ell$ be a non-trivial smooth character.
It defines a non-degenerate character
\begin{equation*}
x \mapsto \psi(x_{1,2}+\dots+x_{n-1,n}) 
\end{equation*}
of $\N$,
the subgroup of upper triangular unipotent matrices of $\G$,
still denoted $\psi$.
Note that~this character takes values in $\zlb$,
and even more precisely in
the group of roots of unity in $\qlb$~whose order is a power of $p$.

Let $\P$ denote the mirabolic subgroup of $G$,
made of all matrices whose last row is $(0 \dots 0\ 1)$.

The representations we will consider will be smooth representations of locally
profinite groups with coefficients in $\ZZ[1/p]$-algebras.

\subsection{}
\label{defint}

A $\qlb$-representation of finite length $\pi$ of $G$ is said to be integral if its 
vector space $V$~con\-tains a $G$-stable $\zlb$-lattice.
(A $G$-stable $\zlb$-lattice is a $G$-stable free $\zlb$-module generated
by~a~ba\-sis of $V$ or, equivalently,
an admissible $\zlb[G]$-module containing a basis of $V$.)

If this is the case, 
and if $L$ is such a $G$-stable $\zlb$-lattice,
the representation~of $G$ on $L\otimes\flb$~is smooth and has finite 
length, and its semisimplification 
does not depend on the choice of $L$~(see \cite{Vigw} Theorem 1).
This semisimplified $\flb$-representation
is called the reduction mod~$\ell$~of $\pi$ and~is denoted $\r_\ell(\pi)$.

An irreducible $\qlb$-representation $\pi$ which embeds in the
parabolic induction of some cuspidal irreducible representation $\rho$
of some Levi subgroup $M$ of $G$
is integral if and only if the central character of $\rho$ is $\zlb$-valued
(see \cite{Vigb} II.4.12, II.4.14
and \cite{Datnu} Proposition 6.7).

\subsection{}
\label{defgen}

In this paragraph, 
$\pi$ is a {generic} irreducible $\qlb$-representation of $\G$,
that is,
its vector~space $V$ carries a
non-zero $\qlb$-linear form ${\it\La}$ such that
${\it\La}(\pi(u)v)=\psi(u){\it\La}(v)$ for all $u\in N$, $v\in V$.
Let
\begin{equation}
\label{contient}
\Ww(\pi,\psi) \subseteq \Ind^\G_\N(\psi)
\end{equation}
denote its Whittaker~model with respect to $\psi$,
where $\Ind^\G_\N$ denotes smooth induction from~$\N$~to $\G$. 
Let $\Kk(\pi,\psi)$ denote the~Ki\-ril\-lov model of $\pi$, that is,
the space of smooth $\qlb$-valued functions on $\P$
which extend to a function in $\Ww(\pi,\psi)$. 
By Kirillov's theory, 
restriction from $\G$ to $\P$~indu\-ces~a~$\P$-equivariant isomorphism 
from $\Ww(\pi,\psi)$ to $\Kk(\pi,\psi)$,
and one has the containments
\begin{equation*}
\ind^{\P}_{\N}(\psi) \subseteq \Kk(\pi,\psi) 
\subseteq \Ind^{\P}_{\N}(\psi)
\end{equation*}
where $\ind^\P_\N$ denotes compact induction from $\N$ to $\P$
(see \cite{BZ}).

\subsection{}

In this paragraph, 
$\pi_1$ and $\pi_2$ are integral generic irreducible 
$\qlb$-representations of $\G$. 
Let $\m_{\ell}$ denote the maximal ideal of $\zlb$.

\begin{prop}
\label{proplocal}
Suppose there are Whittaker functions $\W_1\in\Ww(\pi_1,\psi)$
and $\W_2\in\Ww(\pi_2,\psi)$ with values in $\zlb$ such that:
\begin{enumerate}
\item
$\W_1(1)=\W_2(1)=1$,
\item
$\W_1(g)$ and $\W_2(g)$ are congruent mod $\m_{\ell}$ for all $g\in\G$. 
\end{enumerate}
Then the {reductions mod $\ell$} of $\pi_1$ and $\pi_2$ share a 
generic irreducible component,
such~a generic~ir\-reducible component is unique
and it occurs with multiplicity $1$.
\end{prop}

\begin{proof}
We will need the following result. 

\begin{lemm}
\label{raisin}
Let $\pi$ be an integral generic irreducible $\qlb$-representation of $\G$. 
Then its reduc\-tion mod $\ell$ contains a unique irreducible generic factor,
occuring with multiplicity $1$.
\end{lemm}

\begin{proof}
The existence of an irreducible generic factor 
follows from the fact that any~non-zero li\-near form in
$\Hom_{N}(\pi,\psi)$ is non-zero on any $G$-stable $\zlb$-lattice
of $\pi$.
Its uniqueness follows~for instance from \cite{MSc}
Proposition~8.4~ap\-plied to the representation parabo\-lically induced 
from the cuspidal support of $\pi$.
\end{proof}

Let $i\in\{1,2\}$. 
By \cite{Vigw} Theorem 2,
the $\zlb$-module $L_i$ made of all $\zlb$-valued Whittaker functions in 
$\Ww(\pi_{i},\psi)$ is a $G$-stable $\zlb$-lattice. % of $\Ww(\pi_{i},\psi)$.
Let ${\it\Lambda}_i$ be the $\zlb[\G]$-module~gene\-ra\-ted by
$\W_i$ in $\Ww(\pi_{i},\psi)$. 
It contains a $\qlb$-basis of $\Ww(\pi_{i},\psi)$ since $\pi_i$ is 
irreducible, 
and it~is~con\-tained in $\L_i$. 
It follows that it is a $G$-stable $\zlb$-lattice in $\Ww(\pi_{i},\psi)$.
Let $\M_i$ denote the submodule of $\L_i$ made of all $\m_\ell$-valued
functions.
The containment of $\m_{\ell}{\it\Lambda}_i$ in ${\it\Lambda}_i\cap\M_i$
implies that we have morphisms:
\begin{equation*}
{\it\Lambda}_i\otimes\flb \to {\it\Lambda}_i/({\it\Lambda}_i\cap\M_i)
\simeq ({\it\Lambda}_i+\M_i)/\M_i 
\subseteq \L_i/\M_i \to \Ind^\G_\N(\vartheta\otimes\flb)
\end{equation*}
where the left hand side morphism $\a_i$ is surjective, 
the right hand side morphism $\b_i$ is injective,
and $\vartheta\otimes\flb$ denotes the $\flb$-character of $\N$ obtained 
by reducing $\vartheta$ mod $\m_{\ell}$.

As $\W_{1}$ and $\W_{2}$ are congruent mod $\m_{\ell}$ on $\G$
and take $1$ to $1$, the intersection:
\begin{equation}
\label{interL}
\b_1(({\it\Lambda}_1+\M_1)/\M_1)
\cap
\b_2(({\it\Lambda}_2+\M_2)/\M_2)
\end{equation}
is non-zero in $\Ind^{\G}_{\N}(\vartheta\otimes\flb)$
for it contains the function 
$\b_1(\W_{1}\text{ mod } \M_1) = \b_2(\W_{2}\text{ mod } \M_2)$~and 
the latter is non-zero. 
The socle of \eqref{interL}, denoted ${\it\Sigma}$, 
is made of generic irreducible $\flb$-represen\-tations 
appearing in both the~re\-duc\-tions mod $\ell$ of $\pi_{1}$ 
and $\pi_{2}$.
By Lemma \ref{raisin}, 
the reduction mod $\ell$ of $\pi_{i}$~con\-tains a unique irreducible
gene\-ric factor $\rho_i$. 
The socle ${\it\Sigma}$~is thus~ir\-reducible,~re\-du\-ced~to 
$\rho_i$.
It follows that $\rho_1$ and $\rho_2$ are isomorphic. 
\end{proof}

\subsection{}
\label{par231}

Let $\K$ be the maximal compact subgroup $\GL_n(\Oo_F)$.
In this paragraph,
$\pi$ is an unramified irreducible $\qlb$-representation of $\G$,
that is,
$\pi$ has a non-zero $\K$-fixed vector. 
It defines a~conjuga\-cy class~of~semi\-simple elements in $\GL_n(\qlb)$,
called its Satake parameter.
The characteristic~po\-ly\-nomial of this conjugacy class is denoted $\chi(\pi)$.
This is a polynomial of degree $n$ in $\qlb[X]$.

\begin{lemm}
\label{deWardes}
The unramified representation $\pi$ is integral if and only if
the polynomial $\chi(\pi)$ has all its coefficients in $\zlb$.
\end{lemm}

\begin{proof}
Since $\zlb$ is integrally closed, 
$\chi(\pi)$ has all its coefficients in $\zlb$
if and only if its roots are in $\zlb$,
that is,
if and only if $\pi$ is parabolically induced from an integral
unramified character of the diagonal torus of $\GL_n(F)$.
The lemma then follows from Paragraph \ref{defint}.
\end{proof}

\subsection{}
\label{Wpi}

Now assume that $\pi$ is a generic unramified irreducible 
$\qlb$-representation of $\G$,
and that~$\psi$ is trivial on $\Oo_F^{\phantom{1}}$
but not on $\p_F^{-1}$.
Its Whitta\-ker model $\Ww(\pi,\psi)$ contains a unique
Whittaker~func\-tion~$\W_{\pi}$ such that:
\begin{enumerate}
\item
one has $\W_{\pi}(gk)=\W_{\pi}(g)$ for all $g\in G$ and $k\in K$,
\item
and $\W_{\pi}(1)=1$.
\end{enumerate}
Let us recall the Shintani--Casselman--Shalika formula
\cite{Shintani,Cogdell},
which gives the values of $\W_{\pi}$ at diagonal elements in terms of the
Satake parameter of $\pi$.

Fix a representative
$(\mu_{1},\dots,\mu_{n})\in\overline{\QQ}{}_\ell^{\times n}$ 
of the Satake parameter of $\pi$.
If we write
\begin{equation*}
\chi(\pi) = X^n + c_{1}(\pi) X^{n-1} + \dots + c_{n}(\pi) \in \qlb[X],
\end{equation*}
then
\begin{equation*}
(-1)^{r}c_r(\pi) = \sum\limits_{1\<i_1<\dots<i_r\<n} \mu_{i_1}\dots\mu_{i_r} 
\end{equation*}
for all $r\in\{1,\dots,n\}$.
Let $q$ be the cardinality of the residue field of $F$.
Fix a uniformizer $\w\in\F$
and let ${\it\Delta}$ be the subgroup of $\G$ made of all diagonal~ma\-tri\-ces 
whose eigenvalues are integral~po\-wers of $\w$.
The Iwasawa decomposition $\G=\N{\it\Delta}\K$ shows that $\W_{\pi}$ is
entirely determined by its restriction to ${\it\Delta}$.
Given $a\in\ZZ^{n}$, 
write $\w^a$ for the diagonal matrix whose $i$th eigenvalue is 
$\w^{a_i}$.
 
% Recall 
% that the character $\psi$ is trivial on $\Oo_F$ but not on $\p_F^{-1}$.
One has the formula:
\begin{equation}
\label{formuledeShintani}
\W_{\pi}(\w^a) = 
 q^{\sum\limits_{j=1}^n a_j (j-(n+1)/2) } 
\cdot
\frac{\det((\mu_{j}^{a_l+n-l})_{j,l})}{\prod\limits_{j<l}(\mu_j-\mu_l)} 
\quad
\text{if $a_1\>\dots\>a_n$}
\end{equation}
and $\W_{\pi}$ vanishes at $\w^a$ otherwise. 

\subsection{}

Formula \eqref{formuledeShintani} has the following application.

\begin{prop}
\label{pagniez}
Let $\pi_1$ and $\pi_2$ be integral generic unramified irreducible 
representations~of $\G$.
Assume that the polynomials $\chi(\pi_1)$ and $\chi(\pi_2)$ have
the same reduction mod $\m_\ell$ in $\flb[X]$.~Then
the Whittaker functions $\W_{\pi_1}$ and $\W_{\pi_2}$ are
$\zlb$-valued on $\G$ and one has: 
\begin{equation}
\label{congWhit}
\W_{\pi_1}(g)\equiv\W_{\pi_2}(g) \ {\rm mod}\ \m_\ell
\end{equation}
for all $g\in\G$.
\end{prop}

\begin{proof}
Notice that $\chi(\pi_1)$ and $\chi(\pi_2)$
have coefficients in $\zlb$ by Lemma \ref{deWardes},
thus the reduction of $\chi(\pi_1)$ and $\chi(\pi_2)$ mod $\m_\ell$ 
is well-defined. 

It suffices to prove that $\W_{\pi_1}$ and $\W_{\pi_2}$ are $\zlb$-valued on
${\it\Delta}$
and that the relation \eqref{congWhit} is satisfied for all $g\in{\it\Delta}$.
Fix a representative
$\boldsymbol{\mu}_i=(\mu_{i,1},\dots,\mu_{i,n})
\in\overline{\QQ}{}_\ell^{\times n}$ of the Satake parameter of $\pi_i$
for $i=1,2$.
The scalars $\mu_{i,1},\dots,\mu_{i,n}$ are the roots of $\chi(\pi_i)$.
They are thus in $\zlb$.
The fact~that~the polynomials
$\chi(\pi_1)$ and $\chi(\pi_2)$ and congruent mod $\m_\ell$ 
ensures that, up to reordering, we may~as\-su\-me that 
$\mu_{1,j}$ and $\mu_{2,j}$ are congruent mod $\m_\ell$ for all $j=1,\dots,n$.
The lemma now follows from the Shintani formula \eqref{formuledeShintani}. 
\end{proof}

\subsection{}
\label{platini}

In this paragraph, 
$\pi$ is a {generic} irreducible $\qlb$-representation of $\G$.
We have the following remarkable~integrality~crite\-rion, 
based on \cite{HelmMossINVENT18} and \cite{MaMo}. 

\begin{prop}
\label{lem2}
\label{lemHM}
Let $\pi$ be a generic irreducible $\qlb$-representation of $\G$. 
The following~asser\-tions are equivalent:
\begin{enumerate}
\item 
The representation $\pi$ is integral.
\item
Given any $\zlb$-valued function $f\in\ind_\N^\P(\psi)$,
the Whittaker function in $\Ww(\pi,\psi)$
extending~$f$ is $\zlb$-valued.
\end{enumerate}
\end{prop}

\begin{rema}
Assertion 2 can~be restated as follows:
given any function $W\in\Ww(\pi,\psi)$ whose restriction to $P$ is 
compactly suppor\-ted mod $\N$,
the function $W$ is $\zlb$-valued on $G$ if and only~if it is $\zlb$-valued
on $P$. 
\end{rema}

\begin{proof}
That the first assertion implies the second one
follows from \cite{MaMo} Corollary 4.3
(which gives an even stronger result:
it says that a Whittaker function $W\in\Ww(\pi,\psi)$ is $\zlb$-valued
on~$G$~if and only if it is $\zlb$-valued on $P$).

Let us prove that the second assertion implies the first one.
We will use \cite{HelmMossINVENT18} Theorem 3.2,~which is stated for 
Noetherian al\-ge\-bras over the ring $\WW_\ell$ of Witt vectors of $\flb$. 
Let us explain how it applies to a generic $\qlb$-representation $\pi$
sa\-tisfying Assertion 2.

Let $V$ be the $\qlb$-vector space of $\pi$.
By \cite{Vigb} II.4.9, 
there exists a finite extension $E$ of $\QQ_\ell^{\rm ur}$,~the
maximal~unra\-mi\-fied extension of $\QQ_\ell$ in $\qlb$,
such that $\pi$ is defined over $E$,
that is,
$V$ contains~a~$G$-stable $E$-vector space $V_E$ such that
$V=V_E\otimes_E \qlb$.
Let $\pi_E$ denote the $E$-representation~of~$G$~on $V_E$.
If $\psi_E$ denotes the character $\psi$ considered as being valued in $E$,
then $\pi_E$ is generic~with~res\-pect to $\psi_E$.
Let $K$ be the completion of $E$. 
Then $\pi_K=\pi_E\otimes_E K$ is 
generic with~res\-pect to the character $\psi_K = \psi_E\otimes_E K$.
Since the complete discrete valuation ring $\WW_\ell$
is~iso\-mor\-phic to the completion of the ring of integers
of $\QQ_\ell^{\rm ur}$,
the ring $\Oo$ of integers of $K$ is a~Noe\-the\-rian $\WW_\ell$-algebra.
Let us show that $\pi_K$ satisfies the analogue of Assertion 2 for the ring $\Oo$.

\begin{lemm}
\label{marmite}
Given~any
$\Oo$-valued function $f\in\ind_\N^\P(\psi_K)$,
there exists an $\Oo$-valued function in $\Ww(\pi_K,\psi_K)$
extending~$f$.
\end{lemm}

\begin{proof}
This $f$ can be written $a_1f_1 + \dots + a_rf_r$
with $a_1,\dots,a_r \in \Oo$, and where the functions
$f_1,\dots,f_r \in \ind^G_N(\psi)$ are $\zlb$-valued.
By assumption on $\pi$,
the function $W_i\in\Ww(\pi,\psi)$~ex\-ten\-ding $f_i$ is $\zlb$-valued.
Thus $a_1W_1 + \dots + a_rW_r$ is in $\Ww(\pi_K,\psi_K)$, 
it extends $f$ and it is $\Oo$-valued.
\end{proof}

Let us collect some results from \cite{HelmDMJ16} about the category
$\boldsymbol{{\sf Rep}}_{\WW_\ell}(G)$ of all
smooth $\WW_\ell$-represen\-ta\-tions of $G$. 
This category decomposes into a product of blocks
indexed by~iner\-tial~classes~${\it\Omega}$~of supercuspidal
$\flb$-representations of $G$.
Associated with each block,
there is~its centre $\mathfrak{z}_{\it\Omega}$,
which~is a finitely generated commutative $\WW_\ell$-algebra,
% (see \cite{HelmDMJ16} Theorem 10.8),
and~its \textit{universal co-Whittaker} module $\Ww_{\it\Omega}$, 
which is an admissible $\mathfrak{z}_{\it\Omega}[G]$-module.
% (see \cite{HelmDMJ16} Theorem~6.3).

The representation $\pi_K$ is absolutely irreducible and generic.
It is thus a co-Whittaker $K[G]$-mo\-dule
in the sense of \cite{HelmMossINVENT18} Defini\-tion~2.1.
Also, 
by Schur's lemma,
if ${\it\Omega}$ is the inertial class~asso\-ciated with it, 
the action of the centre $\mathfrak{z}_{\it\Omega}$ on~$\pi_K$
defines a morphism of $\WW_\ell$-algebras
$\chi : \mathfrak{z}_{\it\Omega} \to K$.  
By \cite{HelmDMJ16} Theorem 6.3,
the representation $\pi_K$ is a quotient of
$\Ww_{\it\Omega} \otimes_{\mathfrak{z}_{\it\Omega}} K$.

We now apply \cite{HelmMossINVENT18} Theorem 3.2 to $\pi_K$
(with $A=K$~and $A'=\Oo$),
which says that,
thanks~to Lemma \ref{marmite},
$\chi$ is valued in $\Oo$,
which makes~$\Oo$~into a~$\mathfrak{z}_{\it\Omega}$-al\-gebra.
By \cite{HelmDMJ16} Lemma 6.4 (or more~precisely its proof),
the image $L$ of $\Ww_{\it\Omega} \otimes_{\mathfrak{z}_{\it\Omega}} \Oo$
in $\pi_K$ is an $\Oo$-torsion free {co-Whittaker} $\Oo[G]$-module 
such that $L \otimes_{\Oo}K = \pi_K$.
By \cite{HelmMossINVENT18} Definition~2.1,
this $\Oo[G]$-module $L$ is admissible. 
% that is,
% given any open subgroup $U$ of $G$,
% the $\Oo$-mo\-dule $L^U$ of $U$-fixed vectors is finitely generated,
% thus free.
% It follows (for instance by \cite{Vigb} I.C.5)
% that $L$ is a $G$-stable $\Oo$-lattice of $\pi_K$.
The~repre\-sen\-tation $\pi_K$ is thus $\Oo$-integral.

Fix a parabolic subgroup $Q$ of $G$ with Levi subgroup $M$,
and a cuspidal irreducible representation $\rho$ of $M$ such that $\pi$
embeds in the parabolic induction $\ip^G_Q(\rho)$.
We may and will choose~$E$~so that $\rho$ is also defined 
over $E$:
we thus have an $E$-representation $\rho_E$ such that 
$\rho_E\otimes_E\qlb=\rho$~and
$\pi_E$ embeds in $\ip^G_Q(\rho_E)$.
Thus $\pi_K$ embeds in $\ip^G_Q(\rho_K)$,
where $\rho_K=\rho_E\otimes_EK$~is~cuspidal~and~abso\-lutely~irreducible.
Note that the central character $\omega$ of $\rho_K$ takes values in $E$.

Since $\L$ is admissible,
\cite{Datnu} Proposition 6.7 implies that the Jacquet module $\rp^G_Q(L)$
is admissible, 
thus $\rho_K$ is $\Oo$-integral.
Its central character $\omega$ thus takes values in $\Oo_E$,
thus~the central character of $\rho$ takes values in $\zlb$.
It follows (see Paragraph \ref{defint})
that $\pi$ is integral, as expected.
\end{proof}

\section{Global considerations}
\label{SEC3}

\subsection{}
\label{par41}

Let $k$ be a global field, that is,
either a finite extension of $\QQ$ or the field of rational functions over a 
smooth irreducible projective curve $X$ defined over a finite field of
cardinality $q$.
Let $\AA$~be the ring of ad\`eles of $k$.

Given an integer $n\>2$, 
let $\N=\N_n$ be the subgroup of upper triangular unipotent matrices 
of $\GL_n$
and $P=P_n$ be its mirabolic subgroup,
made of all matrices whose last row is $(0 \dots 0\ 1)$.
More generally,
for $m\in\{0,\dots,n\}$,
let $N_{m,n-m}$ denote the unipotent radical of the
parabolic~sub\-group of $\GL_n$
generated by upper triangular matrices and the Levi
subgroup $\GL_{m} \times \GL_{n-m}$.

Let $\psi:\AA\to\CC^\times$ be a non-trivial continuous
character trivial on $k$. 
It defines in the usual~way a non-degenerate character
of $\N(k)\backslash\N(\AA)$,
namely
\begin{equation*}
u \mapsto \psi(u_{1,2} + \dots + u_{n-1,n})
\end{equation*}
for all $u \in N(\AA)$,
which we still denote by $\psi$.

For any place $v$ of $k$, let $k_v$ denote the completion of $k$ at $v$. 
If $v$ is finite, we write $\Oo_v$ for the ring of integers of $k_v$ and 
$\p_v$ for its maximal ideal.
The character $\psi$ decomposes as 
\begin{equation}
\label{declocpsi}
\psi=\bigotimes_{v}\psi_v
\end{equation}
where $\psi_v$ is a non-trivial continuous character of $k_v$,
trivial on $\Oo_v$ but not on $\p_v^{-1}$ for almost all
finite $v$.

\subsection{}
\label{WhittakerFourrier}

Let us fix a Haar measure $\dd u$ on $\N(k)\backslash\N(\AA)$.
Given a cuspidal irreducible automorphic~re\-pre\-sen\-tation
$\Pi$ of $\GL_n(\AA)$, 
the linear form
\begin{equation}
\h \mapsto \int\limits_{\N(k)\backslash\N(\AA)} \psi(u)^{-1}\h(u)\ \dd u
\end{equation}
on $\Pi$
is known to be well-defined and non-zero
(see \cite{Cogdell} Theorem 1.1 or \eqref{psiwconverse} below).

Associated to $\h\in\Pi$, there is a Whittaker function $\W_\h$ defined by: 
\begin{equation}
\label{psiw}
\W_\h(g) = 
\int\limits_{\N(k)\backslash\N(\AA)} \psi(u)^{-1}\h(ug)\ \dd u
\end{equation}
for all $g\in\GL_n(\AA)$.
The map $\h\mapsto\W_\h$ is a morphism from $\Pi$ to its Whittaker model 
$\Ww(\Pi,\psi)$. 

If we choose for $\dd u$ the Haar measure giving measure $1$ to the compact
group $\N(k)\backslash\N(\AA)$,~one also has a converse expansion:
\begin{equation}
\label{psiwconverse}
\h(g) = \sum\limits_{\g\in\N_{n-1}(k)\backslash\GL_{n-1}(k)} 
\W_\h\left(
\begin{pmatrix}\g&0\\0&1\end{pmatrix}
g\right)
\end{equation}
for all $g\in\GL_n(\AA)$,
with absolute and uniform convergence on compact subsets
(see for instance \cite{GoldfeldHundley} Theorem 13.5.4
or \cite{Cogdell} Theorem 1.1).

\subsection{}
\label{par33}

From now on,
assume that $k$ is a function field of characteristic $p$.
There is thus no~Archi\-me\-dean place. 
Let $\I$ be any sub-$\ZZ[\mu_{p}]$-module of $\CC$,
where $\mu_{p}$ denotes the subgroup of $p$th roots of unity in $\CC$.
Note that $\psi$ takes values in $\mu_p$.

\begin{theo}
\label{GHthm}
Let $\phi:\P_n(\AA)\to\CC$ be a smooth function such that:
\begin{enumerate}
\item 
one has $\phi(\g g)=\phi(g)$ for all $\g\in\P_n(k)$ and all $g\in\P_n(\AA)$, 
\item 
the function $\phi$ is cuspidal in the sense that
\begin{equation*}
\int\limits_{N_{m,n-m}(k) \backslash N_{m,n-m}(\AA)} \phi(ug) \ \dd u = 0
\end{equation*}
for all $g\in\P_n(\AA)$ and all $m\in\{1,\dots,n-1\}$,
\item 
one has
\begin{equation*}
\int\limits_{\N(k)\backslash\N(\AA)} \psi(u)^{-1}\phi(ug)\ \dd u \in \I
\end{equation*}
for all $g\in P_n(\AA)$.
\end{enumerate}
Then $\phi(g)\in\I$ for all $g\in P_n(\AA)$.
\end{theo}

We will prove this theorem by induction on $n\>2$. 
Given a function $\phi$ as in Theorem \ref{GHthm},
it is useful to define a function $W_\phi$ on $\P_n(\AA)$ by setting
\begin{equation*}
W_\phi (g) = \int\limits_{\N(k)\backslash\N(\AA)} \psi(u)^{-1}\phi(ug)\ \dd u 
\end{equation*}
for all $g\in\P_n(\AA)$.
We will often use the fact that, by Assumption 3,
it takes values in $\I$.

\subsection{}

We first treat the case where $n=2$.
We will need the following lemma.

\begin{lemm}
For any smooth functions $f,g \in\Cc^\infty(\AA/k,\CC)$, we have
\begin{equation*}
\int_{\AA/k} f(x) g(x) \ \dd x 
= \sum\limits_{ \g\in k} \ \int_{\AA/k} \psi^{-1} (\g x) f(x) \ \dd x \cdot
\int_{\AA/k} \psi (\g x) g(x) \ \dd x.
\end{equation*} 
\end{lemm}

\begin{proof}
Start with the Fourier expansion formula
\begin{equation*}
f(y) = \sum\limits_{ \g\in k} \ \psi(\g y) \cdot \int_{\AA/k} \psi^{-1} (\g x) f(x) \ \dd x 
\end{equation*}
for $y\in\AA/k$.
Then multiply by $g(y)$ and integrate over $\AA/k$.
\end{proof}

Let $f\in\Cc^\infty(\AA/k,\CC)$.
For any $g\in\P_2(\AA)$,
we thus get
\begin{equation}
\label{iamcharlottesimmons2}
\int_{\AA/k}
\phi \left( \begin{pmatrix}1&u\\0&1\end{pmatrix} g \right)
f(u) \ \dd u 
= \sum\limits_{ \g\in k} {\it\Phi} (\g, g) F (\g) 
\end{equation}
where
\begin{equation*}
{\it\Phi} (\g, g) = \int_{\AA/k} \psi^{-1} (\g u)
\phi \left( \begin{pmatrix}1&u\\0&1\end{pmatrix} g \right)
\ \dd u
\quad
\text{and}\quad 
F (\g) = \int_{\AA/k} \psi (\g u) f(u) \ \dd u.
\end{equation*}
Therefore, we have ${\it\Phi} (0,g)=0$ by cuspidality of $\phi$ and, if $\g\neq0$, we have
\begin{eqnarray*}
{\it\Phi} (\g, g) 
&=& \int_{\AA/k} \psi^{-1} (u) \phi \left(
\begin{pmatrix}1&\g^{-1}u\\0&1\end{pmatrix} g \right) \ \dd u \\
&=& \int_{\AA/k} \psi^{-1} (u) \phi \left(
\begin{pmatrix}\g & u\\0&1\end{pmatrix} g \right) \ \dd u \\
&=& \W_\phi \left( \begin{pmatrix}\g & 0\\0&1\end{pmatrix} g \right)
\end{eqnarray*}
where the first equality follows from the fact that the module of $\g$ is $1$
by the product formula,
and the second one follows from the fact that $\phi$ is $P_2(k)$-invariant.

Now let $U=U(\phi,g)$ be a compact open subgroup of $\AA/k$ such that 
\begin{equation*}
\phi \left(\begin{pmatrix}1&u\\0&1\end{pmatrix} g \right)
=\phi (g)
\quad
\text{for all $u\in U$}.
\end{equation*}
Let $f$ be the characteristic function of $U$.
Thus
\begin{equation*}
F (\g) = \int\limits_{U} \psi(\g u) \ \dd u.
\end{equation*}
On the one hand, we have
\begin{equation*}
\label{intLHS}
\int_{\AA/k}
\phi \left(\begin{pmatrix}1&u\\0&1\end{pmatrix} g \right)
f(u) \ \dd u
= \phi (g) \cdot |\U| 
\end{equation*}
(where $|\U|$ is the volume of $\U$ with respect to $\dd u$).
On the other hand, we have
\begin{equation*}
\int_{\AA/k}
\phi \left( \begin{pmatrix}1&u\\0&1\end{pmatrix} g \right)
f(u) \ \dd u
= \sum\limits_{ \g\in k^\times} \W_\phi \left( \begin{pmatrix}\g &
  0\\0&1\end{pmatrix} g \right) \cdot \int\limits_{U} \psi(\g u) \ \dd u,
\end{equation*}
which gives the identity
\begin{equation*}
\phi (g)
= \sum\limits_{ \g\in k^\times} \W_\phi \left( \begin{pmatrix}\g &
0\\0&1\end{pmatrix} g \right) \cdot 
\int\limits_{U} \psi(\g u) \ \frac {\dd u} {|U|}.
\end{equation*}
As $\W_\phi \left( \begin{pmatrix}\g &0\\0&1\end{pmatrix} g \right) \in \I$, for all $\g \in k^{\times}$,
and 
\begin{equation*} 
\int\limits_{U} \psi(\g u) \ \frac {\dd u} {|U|} = 
\begin{cases} 0 & \text{if $\psi$ is non-trivial on $\g\U$}, \\
1 & \text{otherwise},
\end{cases}
\end{equation*} 
it only remains to prove that the sum over $\g$ is finite,
that is,
there are only finitely many $\g\in k^\times$ such that $\g\U\subseteq\Ker(\psi)$.
Assume that 
\begin{equation*}
U = \prod\limits_{v\in S} \p_v^{m_v}\times \prod\limits_{v\notin S}
\Oo_v^{\phantom{m}}
\end{equation*}
for some finite set $S$ of places of $k$ and some integers $m_v\in\ZZ$.
Recall that we have the~decom\-po\-sition \eqref{declocpsi} of $\psi$.
By taking a bigger $S$ if necessary,
we may (and will) assume that the character $\psi_v$ is trivial on $\Oo_v$
but not on $\p_v^{-1}$ for all $v \notin S$.  
Thus $\g\U\subseteq\Ker(\psi)$ if and only if
$\g\Oo_v\subseteq\Ker(\psi_v)$ for all $v\notin S$ and
$\g\p_v^{m_v}\subseteq\Ker(\psi_v)$ for all $v\in S$.
Equivalently, this means that $\g$ belongs to the space of
$f\in k$, 
considered as rational functions on the curve $X$ defining
the field $k$, such that
\begin{itemize}
\item $f$ has no pole at $v\notin S$, 
\item $f$ has a pole of order $\>-m_v$ at $v\in S$.
\end{itemize} 
The expected finiteness result now follows from the fact that these $f$
form a finite dimensional vector space over $\FF_q$
(see for instance \cite{Stichtenoth} Proposition 1.4.9).

\subsection{}

We now assume that $n\>3$,
and that Theorem \ref{GHthm} has been proved for $\P_{n-1}(\AA)$.

We fix an arbitrary $g\in\P_n(\AA)$ and define a function
$\phi' = \phi'_{g}$ on $\P_{n-1}(\AA)$ by setting
\begin{equation*}
\phi'(h) = 
\int \limits_{(\AA/k)^{n-1}} \psi^{-1}(\n \cdot u)
\phi \left( 
\begin{pmatrix}h&u\\0&1\end{pmatrix} 
g \right) \ \dd u
\end{equation*}
for all $h\in\P_{n-1}(\AA)$, 
where $\n=(0,\dots,0,1)\in k^{n-1}$
and $\a\cdot u = a_1u_1+\dots\a_{n-1}u_{n-1} \in (\AA/k)^{n-1}$
for any $\a\in k^{n-1}$ and $u\in(\AA/k)^{n-1}$.
It has the following properties.

\begin{lemm}
\label{WgProp1}
The function $\phi'$ is cuspidal on $\P_{n-1}(\AA)$, 
that is, 
one has
\begin{equation*}
\int\limits_{N_{m,n-1-m}(k) \backslash N_{m,n-1-m}(\AA)} \phi'(vh) \ \dd v= 0
\end{equation*}
for all $h\in\P_{n-1}(\AA)$ and all $m\in\{1,\dots,n-2\}$.
\end{lemm}

\begin{proof}
Let us fix an $h\in\P_{n-1}(\AA)$ and an $m\in\{1,\dots,n-2\}$.
Then 
\begin{equation}
\label{Ottoline}
\int\limits_{N_{m,n-1-m}(k) \backslash N_{m,n-1-m}(\AA)} \phi'(vh) \ \dd v = 
\int \limits_{(\AA/k)^{n-1}} \psi^{-1}(\n \cdot u) {\it\Omega}_{g,h}(u) 
\ \dd u 
\end{equation}
where
\begin{equation*}
{\it\Omega}_{g,h}(u) = 
\int\limits_{N_{m,n-1-m}(k) \backslash N_{m,n-1-m}(\AA)}
\phi \left( 
\begin{pmatrix} v h&u\\0&1\end{pmatrix} 
g \right) \ \dd v 
\end{equation*}
and the right hand side of \eqref{Ottoline} is equal to
\begin{equation*}
\int \limits_{(\AA/k)^{n-1-m}} \psi^{-1}(\n \cdot u_2) 
\int \limits_{(\AA/k)^{m}} {\it\Omega}_{g,h} \binom {u_1} {u_2} \ \dd u_1 \ \dd u_2 = 
\int \limits_{(\AA/k)^{n-1-m}} \psi^{-1}(\n \cdot u_2) 
{\it\Lambda}_{g,h}(u_2) \ \dd u_2
\end{equation*}
where
\begin{equation*}
{\it\Lambda}_{g,h}(u_2) = 
\int\limits_{N_{m,n-m}(k) \backslash N_{m,n-m}(\AA)} \phi \left( w
\begin{pmatrix} 1&0&0\\0&1&u_2\\0&0&1\end{pmatrix} 
\begin{pmatrix} h&0\\0&1\end{pmatrix} 
g \right) \ \dd w 
\end{equation*}
and this quantity is equal to $0$ 
thanks to the fact that $\phi$ is cuspidal.
\end{proof}

\begin{lemm}
\label{WgProp2}
% For all $g\in\GL_n(\AA)$,
One has
\begin{equation*}
\phi'(\a h)=\phi'(h)
\end{equation*}
for all $\a\in\P_{n-1}(k)$ and $h\in\P_{n-1}(\AA)$.
\end{lemm}

\begin{proof}
Let us fix an $\a\in\P_{n-1}(k)$.
Thanks to the fact that $\phi$ is $\P_n(k)$-invariant, one has
\begin{eqnarray*}
\phi'(\a h) &=& \int \limits_{(\AA/k)^{n-1}} \psi^{-1}(\n \cdot u)
\phi \left( 
\begin{pmatrix}\a h&u\\0&1\end{pmatrix} 
g \right) \ \dd u \\ 
&=& \int \limits_{(\AA/k)^{n-1}} \psi^{-1}(\n\a \cdot u)
\phi \left( 
\begin{pmatrix}h&u\\0&1\end{pmatrix} 
g \right) \ \dd u.
\end{eqnarray*}
Since $\a\in P_{n-1}(k)$,
we get $\n\a=\n$,
thus $\phi'(\a h)=\phi' (h)$.
\end{proof}

\begin{lemm}
\label{WgProp3}
One has
\begin{equation*}
\int \limits_{\N_{n-1}(k) \backslash \N_{n-1}(\AA)}
\psi^{-1}(v) \phi' ( v h ) \ \dd v \in \I
\end{equation*}
for all $h\in \P_{n-1}(\AA)$.
\end{lemm}

\begin{proof}
It suffices to notice that 
\begin{eqnarray*}
\int \limits_{\N_{n-1}(k) \backslash \N_{n-1}(\AA)}
\psi^{-1}(v) \phi' ( v h ) \ \dd v &=&
\int \limits_{\N_{n}(k) \backslash \N_{n}(\AA)}
\psi^{-1}(w) \phi \left( w
\begin{pmatrix}h&0\\0&1\end{pmatrix} 
g \right) \ \dd w \\ &=& 
W_\phi\left( 
\begin{pmatrix}h&0\\0&1\end{pmatrix} 
g \right) 
\end{eqnarray*}
which takes values in $\I$ for all $g\in P_n(\AA)$ and $h\in \P_{n-1}(\AA)$.
\end{proof}

Applying now the inductive hypothesis to the function $\phi'=\phi'_g$,
we deduce that
\begin{equation}
\label{Caitlin}
\phi'_g(h)\in\I, \quad \text{for all $g \in \P_n(\AA)$ and all $h\in\P_{n-1}(\AA)$}.
\end{equation}
We can do even better.

\begin{lemm}
\label{Samantha}
For all $g \in \P_n(\AA)$ and all $g'\in\GL_{n-1}(\AA)$, we have
\begin{equation}
\int \limits_{(\AA/k)^{n-1}} \psi^{-1}(\n \cdot u)
\phi \left( 
\begin{pmatrix}g'&u\\0&1\end{pmatrix} g \right) \ \dd u \in \I.
\end{equation}
\end{lemm}

\begin{proof}
Indeed, we have
\begin{equation*}
\int \limits_{(\AA/k)^{n-1}} \psi^{-1}(\n \cdot u)
\phi \left( 
\begin{pmatrix}g'&u\\0&1\end{pmatrix} g \right) \ \dd u =
\int \limits_{(\AA/k)^{n-1}} \psi^{-1}(\n \cdot u)
\phi \left( 
  \begin{pmatrix}1&u\\0&1\end{pmatrix}
  \begin{pmatrix}g'&0\\0&1\end{pmatrix} g \right) \ \dd u 
\end{equation*}
which is equal to $\phi'_{x}(1)$ with
\begin{equation*}
x = \begin{pmatrix}g'&0\\0&1\end{pmatrix} g \in \P_{n}(\AA).
\end{equation*}
The lemma thus follows from \eqref{Caitlin} applied to the function $\phi'_x$. 
\end{proof}

We now extend $\phi'=\phi'_g$ to $\GL_{n-1}(\AA)$ by setting
\begin{equation*}
\phi'(g') = 
\int \limits_{(\AA/k)^{n-1}} \psi^{-1}(\n \cdot u)
\phi \left( 
\begin{pmatrix}g'&u\\0&1\end{pmatrix} g \right) \ \dd u
\end{equation*}
for all $g'\in\GL_{n-1}(\AA)$.
By Lemma \ref{Samantha},
it takes values in $\I$ on $\GL_{n-1}(\AA)$ for all $g\in\P_{n}(\AA)$. 

Now, 
by Fourier analysis on the compact Abelian group $(\AA/k)^{n-1}$, we have
\begin{eqnarray*}
\phi \left( 
\begin{pmatrix}1&u\\0&1\end{pmatrix} 
g \right) &=&
\sum \limits_{\b\in k^{n-1}} \psi(\b \cdot u) 
\int \limits_{ (\AA/k)^{n-1} } \psi^{-1}(\b \cdot x) \phi \left( 
\begin{pmatrix}1&x\\0&1\end{pmatrix} 
g \right) \ \dd x \\
&=& \sum \limits_{\rho\in P_{n-1}(k) \backslash \GL_{n-1}(k)}
\psi(\n\rho \cdot u) 
\int \limits_{ (\AA/k)^{n-1} } \psi^{-1}(\n \cdot x) \phi \left( 
\begin{pmatrix}\rho&x\\0&1\end{pmatrix} 
g \right) \ \dd x \\
&=& \sum \limits_{\rho\in P_{n-1}(k) \backslash \GL_{n-1}(k)}
\psi(\n\rho \cdot u) \phi'\left( \rho\right).
\end{eqnarray*}
% Taking $g=1$ and
Multiplying by $f(u)$ for some function
$f\in\Cc^\infty((\AA/k)^{n-1},\CC)$ and integrating, we get
\begin{equation*}
\int \limits_{(\AA/k)^{n-1}} 
\phi \left(\begin{pmatrix}1&u\\0&1\end{pmatrix} g \right) f(u) \ \dd u = 
\sum \limits_{\rho\in P_{n-1}(k) \backslash \GL_{n-1}(k)}
\phi'( \rho) \int \limits_{(\AA/k)^{n-1}} 
\psi(\n\rho \cdot u) f(u) \ \dd u.
\end{equation*} 
Now let $U=U(\phi,g)$ be a compact open subgroup of $\AA/k$ such that 
\begin{equation*}
\phi \left( \begin{pmatrix}1&u\\0&1\end{pmatrix} g \right) = \phi (g)
\quad
\text{for all $u\in U^{n-1} \subseteq (\AA/k)^{n-1}$}.
\end{equation*}
Now take for $f$ the characteristic function of $U^{n-1}$.
We get
\begin{equation}
\label{PembrokeLodge}
\phi(g) = 
\sum \limits_{\rho\in P_{n-1}(k) \backslash \GL_{n-1}(k)}
\phi'( \rho) \int \limits_{U^{n-1}}
\psi(\n\rho \cdot u) \ \frac {\dd u} {|U|^{n-1}}.
\end{equation}
For all $\rho$,
we have
\begin{equation*}
\int\limits_{U^{n-1}} \psi(\n\rho \cdot u) \ \frac {\dd u} {|U|^{n-1}} = 
\left\{
\begin{array}{ll} 0 & \text{if $\psi$ is non-trivial on $\n\rho\cdot\U^{n-1}$}, \\
1 & \text{otherwise}.
\end{array}
\right. 
\end{equation*}
A coset $\rho\in P_{n-1}(k) \backslash \GL_{n-1}(k)$
satisfies $\n\rho\cdot\U^{n-1}\subseteq\Ker(\psi)$
if and only if the vector
\begin{equation*}
\b=(\b_1,\dots,\b_{n-1})=\n\rho\in k ^{n-1}
\end{equation*}
satisfies $\b\cdot\U^{n-1}\subseteq\Ker(\psi)$,
that is, $\b_i U\subseteq\Ker(\psi)$ for all $i$. 
But it follows from the case where $n=2$ that there are finitely many $\b_i\in k$ 
such that $\b_i U\subseteq\Ker(\psi)$. 
There are thus finitely many cosets $\rho$ contributing to the sum
\eqref{PembrokeLodge}.

Moreover,
$\phi'(\rho) \in \I$~for all $g\in P_n(\AA)$ and all $\rho\in P_{n-1}(k) 
\backslash \GL_{n-1}(k)$. 
It follows that $\phi(g)\in\I$ for all $g\in P_n(\AA)$.
This finishes the proof of Theorem \ref{GHthm}.

\subsection{}

As in Paragraph \ref{par33},
$k$ is a function field of characteristic $p$ and
$\I$ is a sub-$\ZZ[\mu_{p}]$-module of $\CC$.
Let $Z=Z_n$ denote the centre of $\GL_n$.
We will prove the following result.

\begin{theo}
\label{MSthm}
Let $\phi:\GL_n(\AA)\to\CC$ be a smooth function such that:
\begin{enumerate}
\item 
one has $\phi(\g g)=\phi(g)$ for all $\g\in\GL_n(k)$ and all $g\in\GL_n(\AA)$, 
\item 
the function $\phi$ is cuspidal in the sense that
% one has
\begin{equation*}
\int\limits_{N_{m,n-m}(k) \backslash N_{m,n-m}(\AA)} \phi(ug) \ \dd u = 0
\end{equation*}
for all $g\in\GL_n(\AA)$ and all $m\in\{1,\dots,n-1\}$,
\item 
one has $\phi(g)\in\I$ for all $g\in Z(\AA) P(\AA)$.
\end{enumerate}
Then
\begin{equation*}
\int\limits_{\N(k)\backslash\N(\AA)} \psi(u)^{-1}\phi(ug)\ \dd u \in \I
\end{equation*}
for all $g\in \GL_n(\AA)$.
\end{theo}

We first prove the following lemma. 

\begin{lemm}
\label{density}
The image of $\GL_{n}(k)\Z(\AA)\P(\AA)$ in $\GL_{n}(k)\backslash\GL_n(\AA)$ 
is dense.
\end{lemm}

\begin{proof}
Since the projective space $\PP^{n-1}$ is $k$-isomorphic to 
$\GL_n/\Z\P$, it is equivalent to saying that the image of the 
diagonal embedding of $\PP^{n-1}(k)$ in $\PP^{n-1}(\AA)$ is dense,
that is, $\PP^{n-1}$~satisfies the strong approximation property.
We thus have to prove that, 
given any finite set $S$ of places of $k$ and,
for $v\in S$,
any non-empty open subset $U_v$ of $\PP^{n-1}(k_v)$, 
the intersection
\begin{equation*}
\left(\prod\limits_{v\in S} U_v \times \prod\limits_{v\notin S} 
\PP^{n-1}(\Oo_v) \right) \cap \PP^{n-1}(k)
\end{equation*}
is non-empty.
As $\PP^{n-1}(\Oo_v) = \PP^{n-1}(k_v)$ for all $v$
(given any point $[x_1:\dots:x_{n}] \in \PP^{n-1}(k_v)$,~one
may multiply the $x_i$ by a suitable power of a uniformizer so that
the coordinates are in $\Oo_v$),~this
is equivalent to proving the weak approximation property,
that is, 
proving that there is a point $P = [x_1:\dots:x_{n}] \in \PP^{n-1}(k)$
such that $P \in U_v$ for all $v\in S$.
Restricting to the affine~open~sub\-spa\-ce made of all points whose
last coordinate is non-zero, 
we are reduced to proving 
the weak approximation property
for an affine space. 
This follows from the weak approximation theorem for $k$
(see for instance \cite{Stichtenoth} Theorem 1.3.1). 
\end{proof}

By Assumptions 1 and 3, the function $\phi$ takes values in $\I$
on $\GL_{n}(k)\Z(\AA)\P(\AA)$. 
Since $\phi$~is~lo\-cally constant on $\GL_n(\AA)$, 
it follows from Lemma \ref{density}
that $\phi$ takes values in $\I$ on $\GL_{n}(\AA)$.

We now prove the theorem.
Let us apply the map $\phi\mapsto\W_\phi$ of Paragraph 
\ref{WhittakerFourrier} 
defined by \eqref{psiw}.
This defines a function:
\begin{equation*}
W_\phi(g) =
\int\limits_{\N(k)\backslash\N(\AA)} \psi(u)^{-1}\phi(ug)\ \dd u,
\quad
g\in\GL_{n}(\AA).
\end{equation*}
Let us prove that $\W_\phi$ takes values in $\I$ on $\GL_{n}(\AA)$.
Since $\N(k)\backslash\N(\AA)$ is compact,
there exists~a compact open subgroup $C$ of $\N(k)\backslash\N(\AA)$
such that $\N(k)\backslash\N(\AA)$ is the union of finitely many~$u_i C$
and the function
$u \mapsto \psi(u)^{-1}\phi(ug)$
is constant on these cosets.
This gives us
\begin{equation*}
\W_\phi(g) = |C| \cdot \sum \limits_{i=1}^{r} \psi^{-1}(u_i) \phi(u_i g).
\end{equation*}
It thus remains to prove that $|C|$ is in $\ZZ[\mu_{p}]$.

The measure $\dd u$ gives measure $1$ to the compact group
$\N(k)\backslash\N(\AA)$.
Thus $|C|$ is the index of $C$ in $\N(k)\backslash\N(\AA)$.
Let us prove that $|C|$ is a $p$-power.
For this,
it suffices to prove that $\N(k)\backslash\N(\AA)$ is a pro-$p$-group.
For this,
by d\'evissage,
it suffices to~pro\-ve~that $\AA/k$ is a pro-$p$-group.

By \cite{RaVa} Theorem 5.8, 
there are a finite set $S$ of places of $k$ and~inte\-gers $m_v\>0$ for $v\in S$
such that $\AA=k+U$ for some compact open subgroup 
\begin{equation*}
U = \prod\limits_{v\in S} \p_v^{m_v}\times \prod\limits_{v\notin S} \Oo_v^{\phantom{m}}
\end{equation*}
thus $\AA/k$ is a quotient of $U$.
But $U$ is clearly a pro-$p$-group, as it is a product of pro-$p$-groups.

\section{Modular rigidity}

Let $k$ be a global field as in Paragraph \ref{par41}.
Fix a prime number $\ell$
different from the characteristic of $k$
and a field isomorphism $\ii:\CC\to\qlb$.
Let $\AA=\AA_k$ be its ring of ad\`eles. 

Given any place $v$ of $k$, write $\G_v=\GL_n(k_v)$ and $\N_v=\N(k_v)$, 
and $\P_v=\P(k_v)$ for the~mira\-bolic subgroup of $\G_v$.

Recall that $\m_\ell$ is the maximal ideal of $\zlb$.

\subsection{}

Let us state the following conjecture. 

\begin{conj}
\label{MAINCONJ}
Let $\Pi_1$, $\Pi_2$ be cuspidal automorphic representations of
$\GL_n(\AA)$ with central characters $\Om_1$, $\Om_2$, respectively. 
Let $\iota$ be a field isomorphism from $\CC$ to $\qlb$
for some prime number $\ell$ different from the characteristic of $k$. 
Suppose that:
\begin{enumerate}
\item
the characters $\Om_1\otimes_{\CC}\qlb$ and $\Om_2\otimes_{\CC}\qlb$
are $\zlb$-valued
and congruent mod $\m_\ell$,
\item
there exists a finite set $\SS$ of places of $k$,
containing all Archimedean places and all finite places above $\ell$, 
such that, for all $v\notin\SS$, one has:
\begin{enumerate}
\item
the local components $\Pi_{1,v}\otimes_{\CC}\qlb$ and
$\Pi_{2,v}\otimes_{\CC}\qlb$ are unramified,
\item
the characteristic polynomials of their Satake parameters belong to $\zlb[X]$
and have the same reduction mod $\m_\ell$ in $\flb[X]$.
\end{enumerate}
\end{enumerate}
Let $w$ be a finite place not dividing $\ell$
{such that the representations
$\Pi_{1,w}\otimes_{\CC}\qlb$ and $\Pi_{2,w}\otimes_{\CC}\qlb$ are integral.}
Then the reductions mod $\ell$ of these representations
have a com\-mon generic irreducible compo\-nent,
and such a generic~com\-ponent is unique and occurs with multiplicity $1$.
\end{conj}

\begin{rema}
\label{reductiontowinS}
Note that the case where $w\notin\SS$ is easy.
Indeed,
if $w\notin\SS$,
write
\begin{equation*}
\pi_{1}=\Pi_{1,w}\otimes_{\CC}\qlb,
\quad
\pi_2=\Pi_{2,w}\otimes_{\CC}\qlb.
\end{equation*}
These representations are generic
(as $\Pi_{1,w}$ and $\Pi_{2,w}$ are local components of
cuspidal automorphic represen\-ta\-tions)
and unramified.
For each $i$,
there is thus an unramified character $\omega_i$ of
the diagonal torus~$\T_w$~of $\G_w$
whose parabolic induction is isomorphic to $\pi_i$.
By Lemma \ref{deWardes}, 
these representa\-tions are integral,
that is,
the character $\omega_i$ takes values in
$\overline{\mathbb{Z}}{}^{\times}_{\ell}$.
By \cite{MSautomodl}~Proposition~6.2,
the fact that the characteristic polynomials of their Satake
parameters are congruent~im\-plies
that the reductions mod $\m_\ell$ of $\omega_1$ and $\omega_2$
are conjugate by
the normalizer of~$\T_w$ in $\G_w$.
It follows that $\r_\ell(\pi_1)$ and $\r_\ell(\pi_2)$ are equal,
and the conclusion follows from Lemma \ref{raisin}. 
We will~thus~con\-centrate~on the case where $w\in S$.
\end{rema}

\begin{rema}
\label{redunramnongeneric}
The reader should be aware that there are integral 
unramified irreducible $\qlb$-representations of $\GL_n(k_w)$~whose
Satake parameters have congruent characteristic polynomials, but whose
reductions mod~$\ell$ are unequal.
(For~in\-stance, 
this is the case for the trivial~$\qlb$-charac\-ter
and any integral unramified principal series $\qlb$-representation
whose Satake parameter has~a
characteristic polynomial congruent to that of the Satake
parameter of the trivial $\qlb$-character.)
However, 
this phe\-no\-menon does~not appear~for \textit{generic}
unramified representations.
\end{rema}

\begin{rema}
\label{counterex}
The reductions mod $\ell$ of 
$\Pi_{1,w}\otimes_{\CC}\qlb$ and $\Pi_{2,w}\otimes_{\CC}\qlb$
won't be equal in general for $w\in S$.
Here is an example.
Start with a unitary group $\GG$ of rank $2$
with respect to~a~totally imaginary quadratic extension $l$
of a totally real number field $k$.
Suppose that:
\begin{itemize}
\item 
the group $\GG(k_v)$ is compact for all Archimedean places $v$,
\item
there is a finite place $w$ of $k$ above a prime number $p\neq\ell$ such that 
$\GG(k_w)\simeq\GL_2(k_w)$ and $q$,
the cardinality~of the residue field of $\Oo_w$,
has order $2$ mod $\ell$. 
\end{itemize}
Thanks to our assumption on $q$, 
the $\flb$-representation induced from the trivial $\flb$-character
of a~Borel subgroup of $\GL_2(k_w)$ has length $3$:
its irreducible subquotients are the trivial character,
the unramified character of order $2$ and
a cuspidal subquotient denoted $\rho$
(see \cite{VigGL2}~Théorème~3, Corollaire 5). 
Let $\pi$ be a cuspidal~lift of $\rho$ to $\qlb$, 
that is,
$\pi$ is an integral cuspidal $\qlb$-represen\-tation of $\GL_2(k_w)$
such that $\r_\ell(\pi)=\rho$.
(The existence of such a $\pi$ is granted by \cite{Vigb} III.5.10.)

Now realize $\pi$ as the local component at $w$ of some automorphic
representation $\Pi_1$ of $\GG(\AA_k)$ which is trivial at infinity, and
whose local component at another place $u\neq w$ where $\GG$ splits
is a given cuspidal representation $\eta$ of $\GL_2(k_u)$.

We now follow \cite{MSautomodl} Section 3.
Let $\K_w$ be the maximal compact subgroup $\GL_2(\Oo_w)$
and $\FF_q$~be the residue field of $k_w$.
Let:
\begin{itemize}
\item
$\k_1$ be the inflation to $\K_w$ of the cuspidal irreducible representation
of $\GL_2(\FF_{q})$ occurring in the parabolic induction of the trivial
$\flb$-character of a~Borel subgroup 
(thus the restriction of $\pi$ to $K_w$ contains $\k_1$), 
\item
$\k_2$ be the inflation to $\K_w$ of the Steinberg representation of
$\GL_2(\FF_{q})$.
\end{itemize}
Since the reduction mod $\ell$ of $\k_2$ contains that of $\k_1$, 
we get an automorphic representation $\Pi_2$ of $\GG(\AA_k)$ such that:
\begin{itemize}
\item 
the representation $\Pi_2$ is trivial at infinity,
\item
the representation $\Pi_{2,u}$ is isomorphic to $\eta$, 
\item
the restriction of $\Pi_{2,w}$ to $K_w$ contains $\k_2$
(thus $\Pi_{2,w}$ has non-zero Iwahori fixed vectors),
\item
there is a finite set of places $S$ of $k$,
containing all Archimedean places and $u,w$, 
such that,
for all $v\notin S$,
the representations $\Pi_{1,v}\otimes_\CC\qlb$ and
$\Pi_{2,v}\otimes_\CC\qlb$ are unramified
and the characteris\-tic~polynomials of their Satake parameters
have coefficients in $\zlb$ and have the same reduction.  
\end{itemize}

Using \cite{Labesse},
we transfer $\Pi_1$ and $\Pi_2$ to algebraic regular,
conjugate-selfdual,
cuspidal automorphic re\-presenta\-tions $\widetilde{\Pi}_1$ and 
$\widetilde{\Pi}_2$ of $\GL_2(\AA_l)$.
Applying \cite{MSautomodl} Theorem 8.2, 
and as the~local~trans\-fer at $w$ is the identity
since the group $\GG$ splits at $w$,
we deduce that the representations $\r_\ell(\pi)=\rho$ and
$\r_\ell(\Pi_{2,w})$ share a
generic~ir\-reducible~com\-ponent.
Since $\r_\ell$ commutes to parabolic restriction (by \cite{Datnu} Proposition 6.7),
proving that $\r_\ell(\Pi_{2,w})\neq\rho$~redu\-ces to proving that $\Pi_{2,w}$ is not 
cuspidal. 
But this follows from the fact that $\Pi_{2,w}$ has non-zero~Iwa\-hori fixed 
vectors. 
\end{rema}

\subsection{}
\label{HLTTScholzeVarma}

An instance of Conjecture \ref{MAINCONJ} is provided by
\cite{MSautomodl} Theorem 8.2.
More generally,
the results of~\cite{HLTT,Scholze,Varma} 
imply the conjecture in the case when $k$ is a totally real
or imaginary~CM~number field 
and $\Pi_1,\Pi_2$ are algebraic regular,
by passing to the Galois side and~using~a~density~argu\-ment.
In that case, note that:
\begin{itemize}
\item 
Assumption 1 on central characters is unnecessary,
\item
the representations $\Pi_{i,v} \otimes _\CC \qlb$ are
automatically integral for all finite $v$ not dividing $\ell$.
\end{itemize}

More precisely,
assume that $k$ is a totally real or imaginary CM number~field~and~let
$\Pi_1,\Pi_2$ be algebraic regular cuspidal automorphic representation
of $\GL_n(\AA)$.
Assume that
there exists~a finite set $\SS$ of places of $k$, containing all Archimedean
places and all finite places dividing $\ell$, 
such that,
for all $v\notin\SS$,
one has:
\begin{enumerate}
\item
the local components $\Pi_{1,v}$, $\Pi_{2,v}$ are unramified,
\item
the characteristic polynomials of the conjugacy classes of semisimple
elements in $\GL_n(\qlb)$ associated with $\Pi_{1,v} \otimes_{\CC} \qlb$ and
$\Pi_{2,v}\otimes_{\CC} \qlb$ have coefficients in $\zlb$ and are congruent 
mod $\m_\ell$.
\end{enumerate}
Associated with $\Pi_i$ in \cite{HLTT} and \cite{Scholze}, 
there is a continuous $\ell$-adic Galois~repre\-sen\-ta\-tion
\begin{equation*}
\Sigma_i : \Gal(\overline{\QQ}/k) \to \GL_n(\qlb) 
\end{equation*}
(depending on $\iota : \CC \to \qlb$) for $i=1,2$.
For any finite place $v$ of $k$ not dividing $\ell$,
fix a~decom\-po\-si\-tion subgroup $\Ga_v$ of $\Gal(\overline{\QQ}/k)$.
The Weil-Deligne representation associated with
$\Sigma_i|_{\Ga_v}$ is made of a
smooth $\ell$-adic representation $\rho_{i,v}$ together with
a nilpotent operator on the space of $\rho_{i,v}$.
On the other hand,
the Weil-Deligne representation associated with 
$\Pi_{i,v}\otimes|\det|_v^{(1-n)/2}$ by the local Langlands correspondence
is made of a
semisimple smooth complex representation $\s_{i,v}$ together with
a nilpotent operator on the space of $\s_{i,v}$.
By \cite{Varma},
for any finite place $v$ of $k$ not dividing~$\ell$, 
one has
\begin{equation*}
\rho_{i,v}^{\rm ss} \simeq \s_{i,v}^{\phantom{\rm ss}} \otimes_{\CC} \qlb
\end{equation*}
(where $\rho_{i,v}^{\rm ss}$ stand for the semisimplification of $\rho_{i,v}$).
Arguing as in \cite{MSautomodl} 8.2, 
we deduce that,
for any finite place $v$ of $k$ not dividing $\ell$,
the representations $\Pi_{1,v} \otimes _\CC \qlb$ and
$\Pi_{2,v}\otimes _\CC \qlb$ are integral, 
their reductions mod $\ell$ share a generic irreducible~com\-ponent, 
which occurs with multiplicity $1$.

\subsection{}

\textit{From now on,
and until the end of this section,
$k$ is a function field of characteristic $p$.}~We are going to prove 
Conjecture \ref{MAINCONJ} in this case.
We will actually prove a stronger result. 

\begin{theo}
\label{MAINTHM}
Let $\Pi_1$, $\Pi_2$ be cuspidal automorphic representations of
$\GL_n(\AA)$.
Let $\iota$ be a field isomorphism from $\CC$ to $\qlb$
for some prime number $\ell$ different from $p$. 
Suppose that
there~is a finite set $\SS$ of places of $k$
such that, for all $v\notin\SS$, one has:
\begin{enumerate}
\item
the local components $\Pi_{1,v}\otimes_{\CC}\qlb$ and
$\Pi_{2,v}\otimes_{\CC}\qlb$ are unramified,
\item
the characteristic polynomials of their Satake parameters belong to $\zlb[X]$
and have the same reduction mod $\m_\ell$ in $\flb[X]$.
\end{enumerate}
Let $w$ be a finite place. 
Then:
\begin{itemize}
\item
  the representations $\Pi_{1,w}\otimes_{\CC}\qlb$ and
  $\Pi_{2,w}\otimes_{\CC}\qlb$ are integral
\item
  the reductions mod $\ell$ of these representations
have a com\-mon generic irreducible compo\-nent,
\item
and such a generic~com\-ponent is unique and occurs with multiplicity $1$.
\end{itemize}
\end{theo}

\begin{rema}
\label{MAINTHMRMK}
Note that, by taking a bigger $\SS$, we may (and will) assume that
the character $\psi_v$ is trivial on $\Oo_v$ but not on $\p_v^{-1}$
for all $v\notin\SS$.
\end{rema}

First,
let us prove that,
under the assumptions of Theorem \ref{MAINTHM},
for all $v$,
the central characters of $\Pi_{1,v} \otimes_{\CC}\qlb$
and $\Pi_{2,v} \otimes_{\CC}\qlb$
take values in $\overline{\ZZ}{}_\ell^\times$
and are~con\-gruent mod $\m_\ell$.

\begin{lemm}
\label{Animus}
Let $\chi$ be an automorphic character of $\AA^\times/k^\times$
and $U$ be a subgroup of $\CC^\times$.~As\-sume
that there is a finite set $S$ of places of $k$ such that,
for all $v \notin S$,
the local component $\chi_{v}$~is unramified and
takes values in $U$.
Then, for all $v$,
the character $\chi_{v}$ takes values in $U$.
\end{lemm}

\begin{proof}
If $S$ is empty,
there is nothing to prove.
We thus assume that there is a place $w\in S$.
Let $x\in k_w^\times$.
Define an id\`ele $x' \in \AA^\times$ by
\begin{equation*}
x'_v =
\left\{
\begin{array}{ll} x & \text{if $v=w$}, \\
1 & \text{otherwise}.
\end{array}
\right.
\end{equation*} 
The weak approximation theorem implies that there is a $y\in k^\times$ 
such that $y \in \Ker(\chi_v)$ if $v\in S$ and $v\neq w$,
and $yx \in \Ker(\chi_w)$.
We have
\begin{equation*}
\chi_w(x) = \chi(x') 
= \chi(yx') 
= \chi_w(xy) 
\cdot \prod\limits_{v\in S \atop v\neq w} \chi_v(y) 
\cdot \prod\limits_{v\notin S} \chi_v(y).
\end{equation*}
Thanks to the conditions given by the weak approximation theorem, 
this is equal to the product of $\chi_v(y)$ for all $v\notin S$.
(Note that this is a product of finitely many terms,
since $y$ is a unit in~the ring of integers of $k_v$ for almost 
all $v\notin S$.)
The result follows from the fact that,
for such $v$,
one has $\chi_v(y)\in U$. 
\end{proof}

\begin{prop}
\label{PropAnimus}
Let $\chi_1$ and $\chi_2$ be automorphic characters of $\AA^\times/k^\times$,
and fix a field~iso\-mor\-phism $\iota:\CC\to\qlb$.
Assume there is a finite set $S$ of places of $k$ such that,
for all $v \notin S$:
\begin{enumerate}
\item 
the characters $\chi_{1,v}\otimes_{\CC}\qlb$ and $\chi_{2,v}\otimes_{\CC}\qlb$
are unramified and take values in $\overline{\ZZ}{}_\ell^\times$, 
\item
the reductions mod $\ell$ of these characters are equal. 

\end{enumerate}
Then, for all places $v$,
the characters $\chi_{1,v}\otimes_{\CC}\qlb$ and $\chi_{2,v}\otimes_{\CC}\qlb$
take values in $\overline{\ZZ}{}_\ell^\times$
and are~con\-gruent mod $\m_\ell$.
\end{prop}

\begin{proof}
For Assertion 1 of the proposition, 
apply Lemma \ref{Animus} to $\chi_i$ and $U=\iota^{-1}(\A^\times)$.
For~As\-ser\-tion 2, 
apply Lemma \ref{Animus} to $\chi = \chi_1^{\phantom{1}}\chi_2^{-1}$ 
and $U=1+\iota^{-1}(\m_\ell)$.
\end{proof}

\begin{rema}
\label{RemLaf}
Note that 
Theorem \ref{MAINTHM} follows from \cite{Lafforgue} Th\'eor\`eme VI.9
by a global argument in~the spi\-rit of \S\ref{HLTTScholzeVarma}
(see also \cite{HenniartLemaireMSMF11} IV.1.6).
\end{rema}

\subsection{}
\label{polar}

The remainder of this section is devoted to the proof of Theorem 
\ref{MAINTHM}.
By Remark \ref{reductiontowinS}, 
we may and will assume that $w\in\SS$.

Let $\A$ denote the image of $\zlb$ by $\ii^{-1}$ and $\m$ denote the
ima\-ge~of $\m_\ell$ by $\ii^{-1}$. 
Thus $A$ contains the complex $p$th roots of unity and 
the character $\psi$ of Paragraph \ref{par41} takes values in $\A^\times$.
Notice that $A$ and $\m$ are sub-$\ZZ[\mu_{p}]$-mo\-du\-les of $\CC$.

For any place $v$ of $k$ and $i\in\{1,2\}$,
let $\W_{i,v}$ be a function in the Whittaker model $\Ww(\Pi_{i,v},\psi_v)$ 
satisfying the conditions:
\begin{itemize}
\item 
if $v\notin\SS$,
then $W_{i,v}$
is the unique $\GL_n(\Oo_v)$-invariant function such that
$\W_{i,v}(1)=1$ (see~Pa\-ra\-graph \ref{Wpi}),
\item
if $v\in\SS$,
we fix an arbitrary $\A$-valued function $f_v\in\ind^{\P_v}_{\N_v}(\psi_v)$
and let $W_{i,v}\in\Ww(\Pi_{i,v},\psi_v)$ be the unique function 
extending $f_v$ to $G_v$ (see Paragraph \ref{defgen}),
\item
for all $v\in\SS$ such that $v\neq w$, 
we further assume that $f_v(1)=1$.  
\end{itemize}
For $i\in\{1,2\}$, we consider the global Whittaker function
\begin{equation*}
\W_i = \bigotimes\limits_v \W_{i,v} \in\Ww({\Pi}_i,\psi).
\end{equation*}
For $x\in\P(\AA)$, we thus have
\begin{equation*}
\W_i(x) = \prod\limits_{v \in S} f_v(x_v) \cdot \prod\limits_{v \notin S} W_{i,v}(x_v).
\end{equation*}
It follows from Proposition \ref{pagniez} that 
$\W_1$ and $W_2$ take values in $\A$ 
and $\W_1-\W_2$ takes values in $\m$ on $\P(\AA)$.
Let $\h_i\in\Pi_i$ be the automorphic form corresponding 
to $\W_i$ via \eqref{psiwconverse},
that is: 
\begin{equation*}
\h_i(g) = \sum\limits_{\g\in\N_{n-1}(k)\backslash\GL_{n-1}(k)} 
\W_i\left(
\begin{pmatrix}\g&0\\0&1\end{pmatrix}
g\right)
\end{equation*}
for all $g\in\GL_n(\AA)$.
By Theorem \ref{GHthm},
the functions $\h_1$ and $\h_2$ take values in $\A$ 
and $\h_1 - \h_2$ takes values in $\m$ on $\P(\AA)$.

Thanks to Proposition \ref{PropAnimus}, 
the central characters of $\Pi_{1,v}$
and $\Pi_{2,v}$
take values in $\A^\times$
and are con\-gruent mod $\m$ for all $v$.
It follows that $\h_1$ and $\h_2$ take values in $\A$ 
and $\h_1 - \h_2$ takes values in $\m$ on $\Z(\AA)\P(\AA)$,
where $\Z$ is the~cen\-tre of $\GL_n$.
Applying Theorem \ref{MSthm}, 
we deduce that $\W_1$ and $W_2$ take values in $\A$ 
and $\W_1-\W_2$ takes values in $\m$ on $\GL_{n}(\AA)$.

Now let us consider the place $w$.
For $i=1,2$ and $g\in\G_w\subseteq\G(\AA)$, one has:
\begin{eqnarray*}
\W_i(g) &=& \prod\limits_v \W_{i,v}(g_v) \\
&=& \W_{i,w}(g) \cdot \prod\limits_{v\neq w} \W_{i,v}(1) \\
&=& \W_{i,w}(g). 
\end{eqnarray*}
It follows that $\W_{1,w}$ and $W_{2,w}$ take values in $\A$, 
and that $\W_{1,w} - W_{2,w}$ takes values in $\m$ on~$\G_w$.
We~thus~proved that,
given any $\A$-valued function $f_w\in\ind_{\N_w}^{\P_w}(\psi_w)$,
the functions $\W_{1,w}$ and $W_{2,w}$
extending $f_w$ are $\A$-valued.
Pro\-position 
\ref{lem2} thus implies that $\Pi_{1,w} \otimes_\CC\qlb$ and
$\Pi_{2,w}\otimes_\CC\qlb$~are~in\-te\-gral.
Now assume further that $f_w(1)=1$.
Then Theorem~\ref{MAINTHM} follows from Proposition \ref{proplocal}.

\providecommand{\bysame}{\leavevmode ---\ }

%%%%%%%%%%%%%%%%%%%%%%%%%%%%%%%%%%%%%%%%%%%%%%%%%%%%%%%%%%%%%%%%%%%%%%%%

\end{document}